\documentclass[review]{elsarticle}
\usepackage{lineno}
\usepackage[latin1]{inputenc}
\usepackage{times}
\usepackage{amssymb,mathrsfs, amsmath}
\usepackage{amsmath,amsthm}
\usepackage{amssymb}
\usepackage{latexsym}
\usepackage{amsfonts}
\usepackage{mathrsfs}
\usepackage{epsfig}
\usepackage{hyperref}
\usepackage{graphicx}
\usepackage{graphics}
\usepackage{pst-node}
\usepackage{tikz-cd} 

\newtheorem{thm}{Theorem}[section]

\newtheorem{defn}{Definition}[section]
\newtheorem{prop}{Proposition}[section]

\newtheorem{lem}{Lemma}[section]

\newtheorem{rem}{Remark}[section]

\newtheorem{cor}{Corollary}[section]

\newtheorem{eg}{Example}[section]

\journal{  }

\begin{document}
\begin{frontmatter}

\title{Equivariant cohomology, deformations and central extension of Hom Lie triple systems}

\author{Rinkila Bhutia}
\ead{rbhutia@cus.ac.in}
\address{Sikkim University, Gangtok, Sikkim, 737102, \textsc{India}}

\author{ RB Yadav \fnref{}\corref{mycorrespondingauthor}}
\cortext[mycorrespondingauthor]{Corresponding author}
\ead{rbyadav15@gmail.com, ryadav2@gitam.edu}
\address{GITAM(Demmed to be University), Bengaluru, 562103, \textsc{India}}

\author{Namita Behera\corref{mycoauthor}}
\ead{nbehera@cus.ac.in}

\address{Sikkim University, Gangtok, Sikkim, 737102, \textsc{India}}


\begin{abstract}
In this paper, we study equivariant cohomolgy theory of Hom Lie triple systems. Using this cohomology, we study 1-parameter formal deformation and central extensions of Hom Lie triple systems in the equivariant context.
\end{abstract}

\begin{keyword}
\texttt{Hom-Lie triple system, Group actions, equivariant cohomology, equivariant formal deformations, equivariant central extension}
\MSC[2020] 17A40 \sep 17B10 \sep 17B56 \sep 55U15  \sep 14D15 \sep 55S91 \sep 55N91
\end{keyword}

\end{frontmatter}

\section{Introduction}
A Hom structure on a Lie triple system twists the ternary operation of the system. The concept of Hom Lie algebras, was introduced by Hartwig, Larsson and Silverstrov \cite{Hart}. Since then, algebras with hom type structure has been studied by many authors. Yau \cite{Yau} generalises Lie triple system to define Hom Lie triple system as a particular case of ternary Hom-Nambu algebras.

Algebraic deformation theory was introduced by Gerstenhaber for rings and algebras \cite{MG1, MG2,MG3,MG4, MG5}. This theory studies an object by deforming it into a family of similar objects, depending on a parameter. Kubo and Taniguchi \cite{Kub-Tani} studies deformation theory for Lie triple systems. The notion of 1-parameter formal deformation of Hom Lie triple system has been studied by Ma, Chen and Lin in\cite{Ma-Chen-Lin}. In this, they show that the theory of 1-parameter formal deformation of Hom Lie triple system is governed by the cohomology groups considered for the deformation. They also introduce the notion of central extensions of Hom Lie triple system and establish a correspondence 
between equivalent classes of central extensions of Hom Lie triple systems and the third cohomolgy group.

The idea of formal 1-parameter deformation equipped with a group action was introduced in \cite{gmrb} for associative algebras to give what is known as equivariant deformation. \cite{RBY-NB-RB} studies equivariant formal deformation of Lie triple system.

In this paper, we consider the notions of deformation cohomolgy, central extensions and formal deformation theory of Hom Lie triple system in the equivariant context. The scheme of the paper is as follows:
Section 2 recalls the definition and related concepts of Hom Lie triple system. In Section 3, group actions and equivariant cohomolgy for Hom Lie triple system is studied.  In Section 4, the concept of equivariant central extensions of Hom Lie triple system is introduced. In this section we prove the existence of a one to one correspondence between equivalent classes of equivariant central extensions and the third equivariant cohomology group.
 Section 5 introduces equivariant deformation of Hom Lie triple system. We show that obstructions to equivariant deformation are equivariant cocycles.  In Section 6,  equivalence of deformations and rigidity of Hom Lie triple system are studied in the equivariant context. Finally in Section \ref{HLTSE}, we discuss an example of equivariant deformation of Hom-Lts.

\section{Hom Lie Triple System}

\begin{defn}
   A \textit{ Lie triple system} (\textit{ Lts}) is a vector space $T$ over $k$ with a $k$-linear map $\mu:T\otimes T\otimes T\to T$  satisfying   (if we write $\mu(a\otimes b\otimes c)=[abc]$)
 \begin{equation}\label{LT1}
    [aab]=0,
  \end{equation}
 \begin{equation}\label{LT2}
    [abc]+[bca]+[cab]=0,
  \end{equation}
   \begin{equation}\label{LT3}
    [ab[cde]]=[[abc]de]+[c[abd]e]+[cd[abe]],
  \end{equation}
  for $a,b,c,d,e\in T.$
$[\;]$ is called the ternary operation of the Lie triple system $T$.
\end{defn}

\begin{defn}
A \textit{Hom Lie triple system} (\textit{Hom Lts}) is a vector space $T$ over $k$ with a $k$-linear map $\mu:T\otimes T\otimes T\to T$  and linear maps $\alpha_1, \alpha_2:T \to T$, called the twisted maps, satisfying   
\begin{equation}\label{HLT1}
    [aac]=0,
  \end{equation}
 \begin{equation}\label{HLT2}
    [abc]+[bca]+[cab]=0,
  \end{equation}
   \begin{equation}\label{HLT3}
    [\alpha_1(a) \alpha_2(b)[cde]]=[[abc]\alpha_1(d)\alpha_2(e)]+[\alpha_1(c)[abd]\alpha_2(e)]+[\alpha_1(c)\alpha_2(d)[abe]],
  \end{equation}
  where  a,b,c,d,e are in $T$. We write this as $(T,[~~], \alpha)=(\alpha_1, \alpha_2)$
\end{defn}
We say the \textit{Hom Lts} \textit {T} is multiplicative if
\begin{equation}
  \alpha_1=\alpha_2=\alpha
  \end{equation} and 
 \begin{equation}
   \alpha [abc]=[\alpha(a)\alpha(b)\alpha(c)].
  \end{equation}
Hereon our twisted map $\alpha$ is considered multiplicative and we denote our \textit{Hom Lts} as $(T,[~~],\alpha)$.
\begin{defn}
A morphism $f:(T,[~~ ],\alpha) \to (T',[~~ ]',\alpha')$ of multiplicative \textit{Hom Lts} is a linear map such that 
\begin{equation}
  f [abc]=[f(a)f(b)f(c)]'.
  \end{equation}
 \begin{equation}
  f \alpha=\alpha ' f.
  \end{equation}
\end{defn}
When $\alpha$ is the identity morphism $I_T$, T is a \textit{Lts}.

\begin{eg}\label{eg1}
Let $V$ be a $k$-module and $\mu:V\times V\to k$ be a bilinear map such that $\mu(x,y)=\mu(y,x),~\forall x,y\in V$.
Suppose $\alpha: V \to V$ is a linear map that is invariant with respect to the bilinear map $\mu$  in the sense that $\mu(x,y)=\mu(\alpha(x),\alpha(y))$.
Then for any scalar $\lambda \in k$, $ (V, [~~],\alpha)$
with
$[xyz] = \lambda (\mu(y, z)\alpha(x)-\mu(z, x)\alpha(y))$ is a Hom Lie triple system.
\end{eg}

\begin{eg}\label{eg2}
Let $V=M(p,q),$ be the $k$-module consisting of  all $p\times q$ matrices with entries in a $k$-associative algebra $A$. If $\alpha: A\to A$ is any algebra morphism, then $(V,\alpha \circ [~~], \alpha)$ is a Hom Lie triple system with the ternary 
operation $[\;]$ defined by $[ABC]=(AB^t-BA^t)C+C(B^tA-A^tB)$, for any $A, B, C\in M(p,q)$. Here $A^t$ denotes the transpose of A, for any $A\in M(p,q)$.
\end{eg}
\begin{eg}\label{eg3}
 Suppose $(A, \mu)$ is an associative algebra and $\alpha: A\to A$  is an algebra morphism. Then
$(A, \alpha \circ [~~ ], \alpha)$
is a  Hom Lie triple system, where
$[xyz]=  2[[x, y], z]- [[z, x], y] -[[y, z], x]$. Here$ [~~]$ is the commutator bracket of $\mu$.
\end{eg}

\begin{eg}\label{eg4}
\cite{Yau} Suppose $(A,\mu, \alpha)$ is a Hom associative algebra. Then
$ (A, [~~], \alpha^2)$
is a Hom Lie triple system, where
$[xyz] = \mu(x,y)\alpha(z) -\mu(y,x)\alpha(z)-\mu (z,x)\alpha(y) + \mu(z,y)\alpha(x)$
for $x, y, z\in A$.
\end{eg}

\subsection{Representation of Hom Lie Triple Systems}

\begin{defn}
Let $(T,[~~],\alpha)$ be a \textit{Hom Lts}, V a $ k$ vector space and $A\in$ End(V).
We say that  $V$ is a \textit{module} over  $T$ with respect to $A\in End(V)$  provided that\\
 $E_V := T \oplus V$ possesses the structure of a multiplicative Hom Lie triple system with the twisting map $\beta=\alpha+A$ where $\beta \restriction T= \alpha$ and $\beta \restriction V= A$ such that: \\
(a) $T$ is a Hom Lie triple subsystem of $E_V$,\\
 (b) for $a, b, c \in E_V$, $[a, b, c] \in V$ if any one of a, b, c lies in $V$, and\\
 (c) [a, b, c] = 0 if any two of a, b, c lie in $V$.\\
 We also say that $V$ is a $T$-module with respect to $A$.
\end{defn}
Equivalently we have the following:

\begin{thm} Let $(T,[~~],\alpha)$ be a \textit{Hom Lts}, V an $k$ vector space and $A\in$ End(V). Then $V$ is a $T$-module with respect to $A$  iff there exists a $\theta :T\times T\to End(V)$, a bilinear map, such that $\forall a,b,c,d \in T$
\begin{equation}
\theta ( \alpha (a),\alpha(b))A=A\theta (a,b) \label{HLTS1}
\end{equation}
\begin{equation}
\theta(\alpha(c),\alpha(d))\theta(a,b)-\theta(\alpha(b),\alpha(d))\theta(a,c) - \theta(\alpha(a),[bcd])A+D(\alpha(b),\alpha(c))\theta(a,d)=0 \label{HLTS2}
\end{equation}

\begin{equation}
\theta(\alpha(c),\alpha(d))D(a,b)-D(\alpha(a),\alpha(b))\theta(c,d) +\theta([abc],\alpha(a))A+\theta(\alpha(c),[abd])A=0 \label{HLTS3}
\end{equation}
where $D(a,b)=\theta (b,a)-\theta(a,b)$.\\

We also call $(V,\theta)$, the representation of the \it{Hom Lts T} with respect to $A$.
\end{thm}

\begin{proof}Suppose we have an appropriate $\theta$ satisying \ref{HLTS1}, \ref{HLTS2} and \ref{HLTS3}. Define \\
 $[~~]_V:(T\oplus V)\times(T\oplus V)\times (T\oplus V)\to T\oplus V$ by
$$[(a,u)(b,v)(c,w)]_V=([abc],\theta(b,c)(u)-\theta(a,c)(v)+D(a,b)(w))$$ and the twisted map
$\alpha+A:T\oplus V\to T\oplus V$ by
$$ (\alpha +A)(a,u)=(\alpha(a),A(u))$$

With the above $E_V$ becomes a multiplicative \textit{Hom~Lts} satisfying (a), (b) and (c). Thus $V$ is a $T$ module with respect to $A$.  \\

Conversely suppose  $V$ is a $T$-module with respect to $A$. Then define, \\
$\theta :T\otimes T\to End(V)$ as
$\theta (a\otimes b)(v)=[vab]$, for all $a,b\in T$, $v\in V$. With this $\theta$

\begin{enumerate}
\item 
 $\theta ( \alpha (a),\alpha(b))A(v)=\theta ( \beta (a),\beta(b))A(v)=[A(v) \beta(a) \beta(b)]\\
=[\beta(v) \beta(a) \beta(b)]=\beta[v a b]=A[v a b]=A\theta(a,b)v,$ which gives \ref{HLTS1}
\item \begin{scriptsize}
 $\theta(\alpha(c),\alpha(d))\theta(a,b)v-\theta(\alpha(b),\alpha(d))\theta(a,c)v - \theta(\alpha(a),[bcd])Av+D(\alpha(b),\alpha(c))\theta(a,d)v $\\ \end{scriptsize}
$=[[v a b] \alpha(c) \alpha(d)]-[[vac] \alpha(b) \alpha(d)]-[A(v) \alpha(a) [bcd]]+[\alpha(b) \alpha(c) [vad]]\\
=[[v a b] \beta(c) \beta(d)]-[[vac] \beta(b) \beta(d)]-[\beta(v) \beta(a) [bcd]]+[\beta(b) \beta(c) [vad]]\\
=[\beta(v) \beta(a) [bcd]]-[\beta(v) \beta(a) [bcd]]\\
=0$, which gives \ref{HLTS2}
\item \begin{scriptsize}  $\theta(\alpha(c),\alpha(d))D(a,b)v-D(\alpha(a),\alpha(b))\theta(c,d)v +\theta([abc],\alpha(a))Av+\theta(\alpha(c),[abd])Av $\\ \end{scriptsize}
$=[[a b v] \alpha(c) \alpha(d)]-[\alpha(a) \alpha(b) [vcd]]+[A(v) [abc] \alpha(d)]+[A(v) \alpha(c) [abd]]\\
=[[a b v] \beta(c) \beta(d)]-[\beta(a) \beta(b) [vcd]]+[A(v) [abc] \beta(d)]+[A(v) \beta(c) [abd]] \\
=[\beta(a) \beta(b) [vcd]]-[\beta(a) \beta(b) [vcd]]\\
=0$, which gives \ref{HLTS3}.
\end{enumerate}

\end{proof}

%
%
%
%
%

\subsection{The Cohomology of Hom-Lts}
\begin{defn}
  Let $T$ be a Hom Lie triple system and $V$ be a module over $T$ wrt $A$ represented by $(V,\theta)$. We define an n-Hom cochain as an  $ f\in Hom_k(T^{\otimes (2n+1)},V),$ satisfying the following:
$$Af(x_1,\cdots, x_{2n+1})=f(\alpha(x_1),\cdots, \alpha(x_n)),$$
   $$f(x_1,\cdots, x_{2n-2},x,x,y)=0~~\text{ and }$$
  $$f(x_1,\cdots, x_{2n-2},x,y,z)+f(x_1,\cdots, x_{2n-2},y,z,x)+f(x_1,\cdots, x_{2n-2},z,x,y)=0.$$
For $n\ge 1$, $C^{2n+1}_{\alpha, A}(T;V)$ is the set of all n-Hom cochains.\\
 The coboundary operator $\delta^{2n-1}:C^{2n-1}_{\alpha,A}(T;V)\to C^{2n+1}_{\alpha,A}(T;V)$ is a k-linear map satisfying
  \begin{scriptsize}
  \begin{eqnarray*}
    &&\delta ^{2n-1}f(x_1,\cdots, x_{2n+1})\\
     &=& \theta(\alpha ^{n-1}(x_{2n}),\alpha ^{n-1}(x_{2n+1}))f(x_1,\cdots,x_{2n-1})\\
     &&-\theta(\alpha ^{n-1}(x_{2n-1}),\alpha^{n-1}(x_{2n+1}))f(x_1,\cdots,x_{2n-2},x_{2n})\\
    &&+\sum_{k=1}^{n}(-1)^{k+n}D(\alpha ^{n-1}(x_{2k-1}),\alpha ^{n-1}(x_{2k}))f(x_1,\cdots,\widehat{ x_{2k-1}}\widehat{x_{2k}}, \cdots, x_{2n+1}) \\
     && +\sum_{k=1}^{n}\sum_{j=2k+1}^{2n+1}(-1)^{n+k+1}f(\alpha(x_1),\cdots,\widehat{ x_{2k-1}}\widehat{x_{2k}}, \cdots,[x_{2k-1}x_{2k}x_j],\cdots, \alpha(x_{2n+1})).
 \end{eqnarray*} 
  \end{scriptsize}
\end{defn}
It is routine to verify that the above is a well defined map. With $\delta$ defined as above, we get the following
\begin{thm}
$\delta^{2n+1}\delta^{2n-1}=0.$
\end{thm}

 Thus for {\textit {Hom Lts}} $(T, [~~],\alpha)$, with the above coboundary map, we have a cochain complex with cohomology space denoted by $H_{\alpha.A}(T,V)$.

\section{Group actions and Equivariant Cohomology for \textit{Hom Lts}}
Let $T$ be a \textit{Hom Lts} with its ternary operation $\mu (a\otimes b\otimes c) = [abc]$ and $G$ be a finite group. We say that the group $G$  acts on $T$ from the left if there exists a function
$$\phi:G\times  T \to T\;\;\;\; (g, a)\mapsto \phi(g, a) = ga$$
satisfying the following conditions
\begin{enumerate}
  \item $ex = x$ for all $x \in T$, where $e \in G$ is the group identity.
  \item $g_1(g_2x) = (g_1g_2)x$ for all $g_1, g_2 \in G$ and $x \in T$.
  \item  For every $g \in G$, the left translation $\phi_g = \phi (g, ) : T \to T,$ $a \mapsto ga$ is a linear map.
  \item For all $g \in G$ and $a, b, c\in T$,  $\mu(ga,gb,gc) = g\mu(a, b,c) = g[abc]$, that is, $\mu$ is equivariant with respect to the diagonal action on $T \otimes T\otimes T$.
\item $\alpha(g a)=g\alpha(a)$ $ \forall g\in G$ and $a\in T$
\end{enumerate}
We denote an action as above  by $(G,T)$. We call  the \textit{Hom Lts T}  with an action of a group $G$ as $G$-$Hom~Lts$.

\begin{prop}
  Let $G$ be a finite group and $T$ be a Hom Lie triple system. Then $G$ acts on $T$ if and only if there exists a group homomorphism
  $$ \psi: G \to Iso_{\textit{H-Lts}}(T,T),\;\; g\mapsto  \psi(g) = \phi_g$$
from the group $G$ to the group of Hom Lie triple system isomorphisms from $T$ to $T$.
\end{prop}

\begin{proof}
  For an action $(G,T)$, we define a map $\psi:G\to Iso_{\textit{H-Lts}}(T,T)$ by $\psi(g)=\phi_g.$ One can verify easily that $\psi$ is a  group homomorphism and $\psi(g)$ is a bijective morphism of \textit{Hom Lts}.
  Now, let $\psi:G\to Iso_{\textit{Lts}}(T,T)$ be a group homomorphism. Define a map $G\times T\to T$ by $(g,a)\mapsto \psi(g)(a).$ It can be easily seen that this is an action of $G$ on the \textit{Hom Lts} $T$.
\end{proof}

\begin{defn}
A G-morphism $f:(T,[ ~~], \alpha)\to (T',[ ~~]', \alpha')$ of \textit{G-Hom Lts} is a linear map such that
\begin{equation}
  f [abc]=[f(a)f(b)f(c)]',  \forall a,b,c\in G, 
  \end{equation}
 \begin{equation}
  f \alpha=\alpha ' f.
  \end{equation}
\begin{equation}
  gf=fg, \forall g\in G.
  \end{equation}
\end{defn}

\begin{defn}
  Let $T$ be \textit{G-Hom Lts}. A \textit{G-module} over $T$ is a module $V$ of $T$ with respect to $A$ such that $G$ acts on $V$, and $\theta$ satisfies
$$\theta(ga,gb)(gv)=g(\theta (a,b)v), ~\forall a,b \in T, g \in G ~ \text{and} ~ v \in V.$$

\end{defn}

Define, $\forall n\ge 0$,
\begin{scriptsize}
$$C_{G(\alpha, A)}^{2n+1}(T;V)=\{c\in C_{\alpha,A}^{2n+1}(T;V): c(gx_1,\cdots, gx_{2n+1})=gc(x_1,\cdots, x_{2n+1}), \text{ for all}\; g\in G\}$$\end{scriptsize} 
An element in $C_{G(\alpha,A)}^{2n+1}(T;V)$ is called an invariant (2n+1)-cochain.  Clearly, $C_{G(\alpha,A)}^{2n+1}(T;V)$ is a vector subspace of $C_{\alpha,A}^{2n+1}(T;V)$.

 We have the following lemma.
\begin{lem}\label{rbnb1}
 $c$ is  an invariant (2n-1)-cochain implies that $\delta^{2n-1}(c)$ is an invariant (2n+1)-cochain.
\end{lem}
\begin{proof}
  Let $c \in C_{G(\alpha,A)}^{2n-1}(T;V)$ and $g\in G.$ By definition, we have $$c(gx_1,\cdots,gx_{2n-1})=gc(x_1,\cdots,x_{2n-1}),$$ $\forall (x_1,\cdots,x_{2n-1})\in T^{\otimes(2n-1)}.$ Also,
  \begin{scriptsize}
  \begin{eqnarray}
     &&\delta^{2n-1}(c)(gx_1,\cdots,gx_{2n+1})\notag \\
    &=& \theta(\alpha^{n-1}(gx_{2n}),\alpha^{n-1}(gx_{2n+1}))c(gx_1,\cdots,gx_{2n-1})\notag \\
    &&-\theta(\alpha^{n-1}(gx_{2n-1}),\alpha^{n-1}(gx_{2n+1}))c(gx_1,\cdots,gx_{2n-2},gx_{2n})\notag\\
    &&+\sum_{k=1}^{n}(-1)^{k+n}D(\alpha^{n-1}(gx_{2k-1}),\alpha^{n-1}(gx_{2k}))c(gx_1,\cdots,\widehat{g x_{2k-1}}\widehat{gx_{2k}}, \cdots, gx_{2n+1})\notag \\
     && +\sum_{k=1}^{n}\sum_{j=2k+1}^{2n+1}(-1)^{n+k+1}c(\alpha(gx_1),\cdots,\widehat{ gx_{2k-1}}\widehat{gx_{2k}}, \cdots,[gx_{2k-1}gx_{2k}gx_j],\cdots, \alpha(gx_{2n+1}))\notag\\
     &=&g\delta^{2n-1}(c)(x_1,\cdots,x_{2n+1})
  \end{eqnarray}
    \end{scriptsize}
 Hence, $c \in C_{G(\alpha,A)}^{2n-1}(T;V)$ implies that $\delta^{2n-1}c \in C_{G(\alpha,A)}^{2n+1}(T;V)$.
\end{proof}
This gives us a cochain complex  $(C_{G(\alpha,A)}^{\ast}(T;V), \delta)$ .
We call this cochain complex, \textit{equivariant cochain complex} of \textit{G-Hom Lts} $T$ with coefficients in the $G$-module $V$. The corresponding cohomology is denoted by $H^*_{G(\alpha,A)}(T,V)$. \\
For $V=T$, we denote the cohomology $H^*_{G(\alpha,\alpha)}(T,T)$ by $H^*_{G(\alpha)}(T)$. In sections 5 and 6, we consider the adjoint representation of $T$ with respect to $\alpha$ and the cohomology given by $H^*_{G(\alpha)}(T).$


\section{Equivariant Central Extensions of Hom Lie Triple System}
Consider a \textit{G-Hom Lts} $(T,[~~],\alpha)$. Let $V$ be a $G$-module over $T$ with respect to $A$, with the trivial bilinear map $\theta=0$. Then with the trilinear map, $\mu=0$, $(V,0,A)$ is also a \textit{G-Hom Lts}.\\
A \textit{G-Hom Lts} $(T_c,[~~]_c,\alpha_c)$ is called an \textit{ equivariant central extension} of $(T,[~~],\alpha)$ by $(V,0,A)$, if there exists an exact sequence of \text{G-Hom Lts},
\begin{equation}\label{cenex}
0\to V\overset{i}\to T_c \overset{\pi}\to T\to 0
\end{equation}
and a $G$-morphism $s: T\to T_c$ satisfying $\pi s =id_T$ and $i(v)\subseteq Z(T_c)$, where $Z(T_c)$ is the center of $T_c$ defined as the set consisting of all those $x\in T_c$ such that $[xT_cT_c]_c=0$.\\
Since $i,\pi$ and $s$ are morphisms of \textit{G-Hom Lts}, we have $\alpha_c i=iA$, $\alpha \pi=\pi \alpha_c$, $\alpha_c s=s \alpha $ and $gA=Ag$,$g\alpha_c=\alpha_cg$, $g\alpha=\alpha g$, $\forall g\in G$.\\

Two \textit{central extensions}, $(T_c,[~~]_c,\alpha_c)$ and $(T_c',[~~]_c',\alpha_c')$ of $(T,[~~],\alpha)$ by $(V,0,A)$ are equivalent if there exists a $G$-isomorphism $\phi: T_c\to T_c'$ such that the following diagram commutes
\[\begin{tikzcd}
0\arrow{r} &  V \arrow{r}{i} \arrow{d}{id_v}  &  T_c \arrow{r}{\pi} \arrow{d}{\phi}  &  T\arrow{r}\arrow{d}{id_T} & 0\\
0\arrow{r} &  V \arrow{r}{i'}                          &  T_c\arrow{r}{\pi'}                             &T\arrow{r}                         & 0.
\end{tikzcd}
\]

The following result establishes the relation between equivariant central extension of a \textit{G-Hom Lts} and the cohomology group $H^3_{G(\alpha, A)}(T,V)$.
\begin{thm} The set of equivalent classes of equivariant central extension of \textit{G-Hom Lts} $(T,[~~], \alpha)$ by \textit{G-Hom Lts} $(V,0,A)$ is in one to one correspondence with the cohomology group $H^3_{G(\alpha, A)}(T,V)$.
\end{thm} 
\begin{proof}
 Let \textit{G-Hom Lts} $(T_c,[~~]_c,\alpha_c)$ be the \textit{equivariant central extension} of $(T,[~~],\alpha)$ by $(V,0,A)$, given by the exact sequence 
$0\to V\overset{i}\to T_c \overset{\pi}\to T\to 0$ and the $G$-morphism $s: T\to T_c$.
Define $h:T\times T\times T \to V$ as $ih(x,y,z)=[s(x)s(y)s(z)]_c-s[xyz]$. It can be verified that $h\in C^3_{G(\alpha,A)}(T,V)$ and $\delta^3h=0$. Thus $[h]\in H^3_{G(\alpha, A)}(T,V)$.\\
Further suppose $(T_c,[~~]_c,\alpha_c)$ and $(T_c',[~~]_c',\alpha_c')$ are equivalent equivariant central extensions of of $(T,[~~],\alpha)$ by $(V,0,A)$. We have $G$-isomorphism $\phi:T_c\to T_c'$ and $G$-morphisms $s: T\to T_c$ and $s':T\to T_c'$ satisfying
\[\begin{tikzcd}
0\arrow{r} &  V \arrow{r}{i} \arrow{d}{id_v}  &  T_c \arrow{r}{\pi} \arrow{d}{\phi}  &  T\arrow{r}\arrow{d}{id_T} & 0\\
0\arrow{r} &  V \arrow{r}{i'}                          &  T_c\arrow{r}{\pi'}                             &T\arrow{r}                        & 0.
\end{tikzcd}
\]
where $\phi i=i',\pi=i'\phi$ and $\pi s=\pi' s'=id_{T_c}.$\\
Suppose $h$ and $h'$ be the 3-cocycles in $C^3_{G(\alpha,A)}(T,V)$ corresponding to the central extensions $T_c$ and $T_c'$ respectively. \\
Consider linear map $f:T\to V$ defined as $i'f(x)=s'(x)-\phi s(x),\forall x\in V$.
It can been verified that $if\alpha=i' Af$ and $igf=ifg, \forall g\in G$. Thus we have that $f\in C_{G(\alpha, A)}^3(T,V).$\\
Furher $i'h(x,y,z)=\phi i h(x,y,z)=\phi[s(x)s(y)s(z)]_c-\phi s[xyz]$ and $i(V)\in Z(T_c)$ gives  $[s'(x)s'(y)s'(z)]'_c=\phi[s(x)s(y)s(z)]_c$. Using this we get that $$i'(h'-h)(x,y,z)=-i'f[xyz]=i'\delta^1f(x,y,z).$$ Hence we conclude that $h$ and $h'$ are in the same cohomology class.

Conversely let $h$ be a cocycle in $C^3_{G(\alpha, A)}(T,V)$. Take $T_c=T\oplus V$ with\\
 $[(x,a)(y,b)(z,c)]_c=([xyz],h(x,y,z))$ and $\alpha_c(x,a)=(\alpha(x), A(a)), ~~\forall x,y,z \in T$ and $a,b,c \in V$.
Then $(T_c,[~~]_c,\alpha_c)$ forms a \textit{Hom Lts}.\\
With G-action $G\times T\oplus V\to T\oplus V, ~ (g,(x,a))\mapsto (gx,ga) $, $(T_c,[~~]_c,\alpha_c)$ becomes a \textit {G-Hom Lts}.\\
Taking $i:V\to T_c$, $a\mapsto (0,a)$; $\pi:T_c\to T$, $(x,a)\mapsto x$ and  $s:T\to T_c$, $x\mapsto (x,0)$ we get that $(T_c,[~~]_c,\alpha_c)$ forms an equivariant central extension of $(T,[~~], \alpha)$ by $(V,0,A)$.

Suppose $h$ and $h'$ are cocycles in the same cohomology class. Then $h-h'=\delta^1f$ for some $ f \in C^3_{G(\alpha,A)}(T,V)$  i.e. $(h'-h)(x,y,z)=-f[xyz], \forall x,y,z \in T$.\\
Let $(T_c,[~~]_c,\alpha_c)$ and $(T_c',[~~]'_c,\alpha_c')$ be the equivariant central extension of $(T,[~~], \alpha)$ by $(V,0,A)$ defined as above, using $h$ and $h'$ respectively.\\
The map $\phi : T_c\to T_c'$ defined as $(x,a)\mapsto (x, a-f(x))$ gives a G-isomorphism between the two equivariant central extensions.
\end{proof}


\section{ Equivariant deformation of a Hom Lie triple system}
\begin{defn}\label{rb2}

Let $T$ be a Hom Lie triple system with an action of $G$. We denote the space of all formal power series with coefficients in $T$ by $T[[t]]$.  An \textit{ equivariant formal one-parameter deformation} of a \textit{G-Hom Lts} $T$ is a $k[[t]]$-linear map  $$\mu_t : T[[t]]\otimes T[[t]]\otimes T[[t]]\to T[[t]]$$ satisfying the following properties:
\begin{itemize}
  \item[(a)]  $\mu_t(a,b,c)=\sum_{i=0}^{\infty}\mu_i(a,b,c) t^i$, for all $a,b,c\in T$, where $\mu_i:T\otimes T\otimes T\to T$ are k-linear and $\mu_0(a,b,c)=\mu(a,b,c)=[abc]$ is the original ternary operation on T.
  \item [(b)] For every $g\in G$, $$\mu_i(ga,gb,gc)=g\mu_i(a,b,c), \;\;\forall a,b,c\in T,$$ for every $i\ge 0.$ This is equivalent to saying that $\mu_i\in Hom_k^G(T\otimes T\otimes T,T),$ for all $i\ge 0$
      \item[(c)]

\begin{equation}\label{DLT0}
    \mu_t(\alpha(a),\alpha(b),\alpha(c))=\alpha \mu_t(a,b,c),
  \end{equation}
  \begin{equation}\label{DLT1}
    \mu_t(a,a,b)=0,
  \end{equation}
 \begin{equation}\label{DLT2}
    \mu_t(a,b,c)+\mu_t(b,c,a)+\mu_t(c,a,b)=0,
  \end{equation}
   \begin{eqnarray}\label{DLT3}
    \mu_t(\alpha(a),\alpha(b),\mu_t(c,d,e))=\mu_t(\mu_t(a,b,c),\alpha(d),\alpha(e))+\notag\\
\mu_t(\alpha(c),\mu_t(a,b,d),\alpha(e))+\mu_t(\alpha(c),\alpha(d),\mu_t(a,b,e)),
  \end{eqnarray}
  for all  $a,b,c,d,e\in T$
\end{itemize}

The equations  \ref{DLT0},\ref{DLT1},\ref{DLT2} and \ref{DLT3}  are equivalent to following equations, respectively:
\begin{equation}\label{rbeqn0}
     \mu_r(\alpha(a),\alpha(b),\alpha(c))=\alpha \mu_r(a,b,c) \;\text{for all}\; a,b,c\in T, \;r\ge 0.
  \end{equation}
\begin{equation}\label{rbeqn1}
  \mu_r(a,a,b)=0, \;\text{for all}\; a,b\in T, \;r\ge 0.
  \end{equation}
  \begin{equation}\label{rbeqn2}
    \mu_r(a,b,c)+\mu_r(b,c,a)+\mu_r(c,a,b)=0, \;\text{for all}\; a,b,c\in T, \;r\ge 0.
   \end{equation}
   \begin{eqnarray}\label{rbeqn3}
      &&\sum_{i+j=r}\mu_i(\alpha(a),\alpha(b),\mu_j(c,d,e))\notag\\
      &=& \sum_{i+j=r}\{\mu_i(\mu_j(a,b,c),\alpha(d),\alpha(e))+\mu_i(\alpha(c),\mu_j(a,b,d),\alpha(e)) \notag\\
       && +\mu_i(\alpha(c),\alpha(d),\mu_j(a,b,e))\}; \;\text{for all}\; a,b,c,d,e\in T,\;  r\ge 0
   \end{eqnarray}

\end{defn}

 Now we define equivariant formal deformations  of finite order.

 \begin{defn}\label{rb3}
Let $T$ be a \textit{Hom Lts} with an action of $G$.  An \textit{ equivariant formal one-parameter deformation of order n}  of a \textit{G- Hom Lts} $T$ is a $k[[t]]$-linear map  $$\mu_t : T[[t]]\otimes T[[t]]\otimes T[[t]]\to T[[t]]$$ satisfying the following properties:
\begin{itemize}
  \item[(a)]  $\mu_t(a,b,c)=\sum_{i=0}^{n}\mu_i(a,b,c) t^i$, for all $a,b,c\in T$, where $\mu_i:T\otimes T\otimes T\to T$ are $k$-linear and $\mu_0(a,b,c)=\mu(a,b,c)=[abc]$ is the original ternary operation on T.
  \item [(b)] For every $g\in G$, $$\mu_i(ga,gb,gc)=g\mu_i(a,b,c), \;\;\forall a,b,c\in T,$$ for every $i\ge 0.$ This is equivalent to saying that $\mu_i\in Hom_k^G(T\otimes T\otimes T,T),$ for all $i\ge 0$
      \item[(c)]
\begin{equation}\label{FDLT0}
     \mu_t(\alpha(a),\alpha(b),\alpha(c))=\alpha \mu_t(a,b,c),
  \end{equation}
  \begin{equation}\label{FDLT1}
    \mu_t(a,a,b)=0,
  \end{equation}
 \begin{equation}\label{FDLT2}
    \mu_t(a,b,c)+\mu_t(b,c,a)+\mu_t(c,a,b)=0,
  \end{equation}
   \begin{eqnarray}\label{FDLT3}
\mu_t(\alpha(a),\alpha(b),\mu_t(c,d,e))&=&\mu_t(\mu_t(a,b,c),\alpha(d),\alpha(e))+\mu_t(\alpha(c),\mu_t(a,b,d),\alpha(e))\notag\\
&&+\mu_t(\alpha(c),\alpha(d),\mu_t(a,b,e)),
   \end{eqnarray}
  for all  $a,b,c,d,e\in T$
\end{itemize}
\end{defn}

\begin{rem}\label{rbrem1}
  \begin{itemize}
    \item For $r=0$, conditions \ref{rbeqn0}-\ref{rbeqn3} are equivalent to the fact that $T$ is a Hom Lie triple system.
    \item For $r=1$, conditions  \ref{rbeqn0}-\ref{rbeqn3} are equivalent to saying that $\mu_1$ is a 3-cocycle in  $C^3_{G(\alpha)}(T)$ . In general, for $r\ge 0$, $\mu_r$ is just a 3-cochain in $C^3_{G(\alpha)}(T).$
  \end{itemize}
\end{rem}

\begin{defn}
  The 3-cochain  $\mu_1$ in $C^3_{G(\alpha)}(T)$ is called  \textit{infinitesimal} of the equivariant  deformation $\mu_t$. In general, if $\mu_i=0,$ for $1\le i\le n-1$, and $\mu_n$ is a nonzero cochain in  $C^3_{G(\alpha)}(T)$, then $\mu_n$ is called \textit{n-infinitesimal} of the equivariant deformation $\mu_t$.
\end{defn}
\begin{prop}
  The infinitesimal   $\mu_1$ of the equivariant deformation  $\mu_t$ is a 3-cocycle in $C^3_{G(\alpha)}(T).$ In general, n-infinitesimal  $\mu_n$ is a 3-cocycle in $C^3_{G(\alpha)}(T).$
\end{prop}
\begin{proof}
  For n=1, proof is obvious from the Remark \ref{rbrem1}. For $n>1$, proof is similar.
\end{proof}

We can write Equations  \ref{rbeqn0}, \ref{rbeqn1}, \ref{rbeqn2} and \ref{rbeqn3} for $r=n+1$ using the definition of coboundary $\delta$ as
\begin{equation}\label{rbeqn4}
   \mu_{n+1}(\alpha(a),\alpha(b),\alpha(c))=\alpha \mu_{n+1}(a,b,c),
  \end{equation}
\begin{equation}\label{rbeqn5}
  \mu_{n+1}(a,a,b)=0,
 \end{equation}
  \begin{equation}\label{rbeqn6}
    \mu_{n+1}(a,b,c)+\mu_{n+1}(b,c,a)+\mu_{n+1}(c,a,b)=0,
   \end{equation}

   \begin{eqnarray}\label{rbeqn7}
      \delta \mu_{n+1}(a,b,c,d,e)
      &=& \sum_{\substack{ i+j=n+1\\i,j>0}}\mu_i(\alpha(a),\alpha(b),\mu_j(c,d,e))\nonumber\\&&-\sum_{\substack{i+j=n+1\\i,j>0}}\{\mu_i(\mu_j(a,b,c),\alpha(d),\alpha(e))\nonumber\\
      &&+ \mu_i(\alpha(c),\mu_j(a,b,d),\alpha(e)) 
        +\mu_i(\alpha(c),\alpha(d),\mu_j(a,b,e))\}\notag\\
   \end{eqnarray}
 for all  $a,b,c,d,e\in T$.

Define a 5-cochain $F_{n+1}$ in $C_\alpha^5(T)$ as
\begin{align*}
&F_{n+1}(a,b,c,d,e)\notag\\
&= \sum_{\substack{i+j=n+1\\i,j>0}}\mu_i(\alpha(a),\alpha(b),\mu_j(c,d,e))-\sum_{\substack{i+j=n+1\\i,j>0}}\{\mu_i(\mu_j(a,b,c),\alpha(d),\alpha(e))\notag\\&+ \mu_i(\alpha(c),\mu_j(a,b,d),\alpha(e))
        +\mu_i(\alpha(c),\alpha(d),\mu_j(a,b,e))\}\\
\end{align*}

\begin{lem}\label{Obsl1}
  The 5-cochain $F_{n+1}$ is invariant, that is $F_{n+1}\in C_{G(\alpha)}^5(T).$
\end{lem}
\begin{proof}
 To prove that $F_{n+1}$ is invariant we show that  $$F_{n+1}(ga,gb,gc,gd,ge)=gF_{n+1}(a,b,c,d,e)$$  for all $a,b,c,d,e\in T$ . From Definition \ref{rb2}, we have $$\mu_i(ga,gb,gc)=g\mu_i(a,b,c),$$ for all $a,b,c\in T.$
 So, we have, for all $a,b,c,d,e\in T$,
 \begin{scriptsize}
 \begin{align*}
   & F_{n+1}(ga,gb,gc,gd,ge) \\
  & = \sum_{\substack{i+j=n+1\\i,j>0}}\mu_i(\alpha(ga),\alpha(gb),\mu_j(gc,gd,ge))-\sum_{\substack{i+j=n+1\\i,j>0}}\{\mu_i(\mu_j(ga,gb,gc),\alpha(gd),\alpha(ge))\\&+ \mu_i(\alpha(gc),\mu_j(ga,gb,gd),\alpha(ge))+\mu_i(\alpha(gc),\alpha(gd),\mu_j(ga,gb,ge))\} \\
   &=  \sum_{\substack{i+j=n+1\\i,j>0}}\mu_i(d\alpha(a),g\alpha(b),g\mu_j(c,d,e))-\sum_{\substack{i+j=n+1\\i,j>0}}\{\mu_i(g\mu_j(a,b,c),g\alpha(d),g\alpha(e))\\
   &+ \mu_i(g\alpha(c),g\mu_j(a,b,d),g\alpha(e)) +\mu_i(g\alpha(c),g\alpha(d),g\mu_j(a,b,e))\} \\
    &=  g\sum_{\substack{i+j=n+1\\i,j>0}}\mu_i(\alpha(a),\alpha(b),\mu_j(c,d,e))-g\sum_{\substack{i+j=n+1\\i,j>0}}\{\mu_i(\mu_j(a,b,c),\alpha(d),\alpha(e))\\
    &+ \mu_i(\alpha(c),\mu_j(a,b,d),\alpha(e))  +g\mu_i(\alpha(c),\alpha(d),\mu_j(a,b,e))\} \\
   &=gF_{n+1}(a,b,c,d,e).
 \end{align*}
 \end{scriptsize}
 So we conclude that $F_{n+1}\in C_{G(\alpha)}^5(T).$
\end{proof}

\begin{defn}
  The 5-cochain $F_{n+1}\in C_{G(\alpha)}^5(T)$ is called  $(n+1)th$ \textit{obstruction cochain} for extending a given equivariant deformation of order n to an equivariant deformation of $T$ of order $(n+1)$. We denote $F_{n+1}$ by $Ob_{n+1}(T)$
\end{defn}

 By using Lemma \ref{Obsl1} and \cite{Kub-Tani}, we have the following result.
\begin{thm}\label{Obst-cocyc}
The $(n+1)$th obstruction cochain $Ob_{n+1}(T)$ is a 5-cocycle.
\end{thm}

\begin{thm}
Let $\mu_t$ be an   equivariant deformation of $T$ of order n. Then $\mu_t$ extends to an equivariant deformation of order $n+1$ if and only if cohomology class of $(n+1)$th obstruction $Ob_{n+1}(T)$  vanishes.
\end{thm}
\begin{proof}
  Suppose that an equivariant deformation $\mu_t$ ,of $T$, of order n extends to an equivariant deformation of order $n+1$. This implies that \ref{rbeqn0}, \ref{rbeqn1}, \ref{rbeqn2} and \ref{rbeqn3} are satisfied for $r=n+1.$  Observe that this implies $Ob_{n+1}(T)=\delta^3\mu_{n+1}$. So cohomology class of  $Ob_{n+1}(T)$ vanishes. Conversely, suppose that  cohomology class of  $Ob_{n+1}(T)$ vanishes, that is $Ob_{n+1}(T)$ is a coboundary. Let
  $$ Ob_{n+1}(T)=\delta^3\mu_{n+1},$$
  for some 3-cochain  $\mu_{n+1}\in C^3_{G(\alpha)}(T).$ Take
  $$\tilde{\mu_t}=\mu_t+\mu_{n+1}t^{n+1}.$$
  Observe that $\tilde{\mu_t}$ satisfies \ref{rbeqn0}, \ref{rbeqn1},\ref{rbeqn2} and \ref{rbeqn3} for $0\le r\le n+1$. So  $\tilde{\mu_t}$ is an equivariant extension of $\mu_t$ and is of order $n+1$.
\end{proof}

\begin{cor}
  If $H^5_{G(\alpha)}(T)=0$, then every 3-cocycle in $C^3_{G(\alpha)}(T)$ is an infinitesimal of some equivariant deformation of $T.$
\end{cor}
\begin{proof}
  Let $\mu_1$ be a 3-cocycle in $C^3_{G(\alpha)}(T)$. This implies that the   conditions  \ref{rbeqn0}-\ref{rbeqn3} are satisfied for $r=1$. Thus $\mu_t(a,b,c)=\mu_0(a,b,c)+\mu_1(a,b,c)t$ is an  equivariant deformation of $T$ of order 1. By the Theorem \ref{Obst-cocyc}, 
\begin{eqnarray*}
 F_2(a,b,c,d,e)&=& \mu_1(\alpha(a),\alpha(b),\mu_1(c,d,e))-\mu_1(\mu_1(a,b,c),\alpha(d),\alpha(e))\\
 &&- \mu_1(\alpha(c),\mu_1(a,b,d),\alpha(e)) -\mu_1(\alpha(c),\alpha(d),\mu_1(a,b,e))
\end{eqnarray*}  
   is a 5-cocycle. Since $H^5_{G(\alpha)}(T)=0$, there exists a 3-cochain $\mu_2$ in $C^3_{G(\alpha)}(T)$ such that  
\begin{eqnarray*} 
  \delta\mu_2(a,b,c,d,e)&=& \mu_1(\alpha(a),\alpha(b),\mu_1(c,d,e))-\mu_1(\mu_1(a,b,c),\alpha(d),\alpha(e))\\ &&- \mu_1(\alpha(c),\mu_1(a,b,d),\alpha(e)) -\mu_1(\alpha(c),\alpha(d),\mu_1(a,b,e)).
\end{eqnarray*}   
   This implies that  conditions  \ref{rbeqn0}-\ref{rbeqn3} are satisfied for $r=2$. Thus  $$\mu_t(a,b,c)=\mu_0(a,b,c)+\mu_1(a,b,c)t+\mu_2(a,b,c)t^2$$ is an  equivariant deformation of $T$ of order 2. Using similar arguments we can extend an equivariant deformation of $T$ of order n to an equivariant deformation of order $n+1$. This gives a sequence $\{\mu_n\}$ of 3-cochains in $C^3_{G(\alpha)}(T)$ such that $\mu_t(a,b,c)=\sum_{i=0}^{\infty}\mu_i(a,b,c) t^i$ satisfies the   conditions  \ref{rbeqn0}-\ref{rbeqn3}. Hence $\mu_t$ is an equivariant deformation of $T$ and $\mu_1$ is an infinitesimal of $\mu_t.$
\end{proof}

\section{Equivalence of equivariant deformations and rigidity }
Let   $\mu_t$  and $\tilde{\mu_t}$ be two equivariant deformations of $T$. An equivariant formal isomorphism from the equivariant deformations $\mu_t$ to $\tilde{\mu_t}$ of a \textit{Hom-Lts} $T$ is a $k[[t]]$-linear $G$-automorphism $\Psi_t:T[[t]]\to T[[t]]$ of the  form  $\Psi_t=\sum_{i\ge 0}\psi_it^i$, where
\begin{itemize}
\item each $\psi_i$ is an equivariant $k$-linear map $T\to T$, $\psi_0(a)=a$, for all $a\in T$
\item  $\tilde{\mu_t}(\Psi_t(a),\Psi_t(b),\Psi_t(c))=\Psi_t\mu_t(a,b,c),$ for all $a,b,c\in T$ and
\item $\psi_t\alpha =\alpha \psi_t$.
\end{itemize}

\begin{defn}
  Two equivariant deformations $\mu_t$  and $\tilde{\mu_t}$ are said to be \textit{equivalent} if there exists an equivariant formal isomorphism  $\Psi_t$ from $\mu_t$ to  $\tilde{\mu_t}$.
\end{defn}
Equivariant formal isomorphism on the collection of all  equivariant deformations of a \textit{Hom-Lts} $T$ is an equivalence relation.
\begin{defn}
  Any equivariant deformation of $T$ that is equivalent to the deformation $\mu_0$ is said to be a \textit{ trivial deformation}.
\end{defn}

\begin{thm}
  The cohomology class of the infinitesimal of an equivariant  deformation $\mu_t$ of  a \textit{Hom-Lts} $T$ is determined by the equivalence class of $\mu_t$.
\end{thm}
\begin{proof}
  Let  $\Psi_t$ from  $\mu_t$ to  $\tilde{\mu_t}$ be an equivariant  formal  isomorphism. So, we have  $\tilde{\mu_t}(\Psi_ta,\Psi_tb,\Psi_tc)=\Psi_t\circ \mu_t(a,b,c),$ and $\psi_t\alpha =\alpha \psi_t$ for  all $a,b,c\in T$. This implies that $(\mu_1-\tilde{\mu_1})(a,b,c)=[\psi_1abc]+[a\psi_1bc]+[ab\psi_1c]-\psi_1[abc]$. So we have $\mu_1-\tilde{\mu_1}=\delta^1\psi_1$ This completes the proof.
\end{proof}
\begin{defn}
  An equivariant \textit{Hom-Lts} $T$ is said to be \textit{rigid} if every deformation of $\mu_t$ of $T$  is trivial.
\end{defn}
\begin{thm}\label{rb-100}
  A non-trivial equivariant deformation of a \textit{Hom-Lts} is equivalent to an equivariant  deformation whose n-infinitesimal is not a coboundary, for some $n\ge 1.$
\end{thm}
\begin{proof}
  Let $\mu_t$ be an equivariant deformation of  a \textit{Hom Lts} $T$ with n-infinitesimal $\mu_n$, for some $n\ge 1.$ Assume that there exists a 1-cochain $\psi\in C_{G(\alpha)}^1(T)$ with $\delta^1\psi=\mu_n.$  Take $\Psi_t=Id_T+\psi t^n$. Define $\tilde{\mu_t}=\Psi_t\circ \mu_t\Psi_t^{-1}$. Clearly, $\tilde{\mu_t}$ is an equivariant deformation of $T$ and $\Psi_t$ is an   equivariant  formal isomorphism from $\mu_t$  to  $\tilde{\mu_t}$. For $u,v,w\in T,$ we have  $\tilde{\mu_t}(\Psi_tu,\Psi_tv, \Psi_tw)=\Psi_t(\mu_t(u,v,w)),$ which implies $\tilde{\mu_i}=0,$ for $1\le i\le n.$ So $\tilde{\mu_t}$ is equivalent to the given deformation and $\tilde{\mu_i}=0,$ for $1\le i\le n.$ We can repeat the argument to get rid off any infinitesimal that is a coboundary. So the process must stop if the deformation is nontrivial.
\end{proof}
As a consequence of the above Theorem \ref{rb-100}, we have the following corollary.
\begin{cor}
  If $H_{G(\alpha)}^3(T)=0,$ then $T$ is rigid.
\end{cor}

\section{Examples of Equivariant Deformation of Hom-Lts}\label{HLTSE}
Now, we see one example of equivariant deformation of order $1$. 
\begin{eg}
 Let $T$ be a vector space generated by $\{e_1,e_2\}$.
 Let the only non-zero triple brackets be defined as 
$$[e_1,e_2, e_2]=e_1~\text{and} ~[e_2,e_1, e_2]=-e_1.$$
Consider the twisted map $\alpha: T\to T$ be the linear map defined as
$$\alpha(e_1)=e_1~\text{and} ~\alpha(e_2)=-e_2.$$
\textbf{Claim}: $(T,[~],\alpha)$ is a multiplicative $Hom$-$Lts$.
\begin{itemize}
	
	\item $[a,a,b]=0$ for all $a,b\in T$ by definition

\item $[e_1,e_2,e_2]+[e_2,e_2,e_1]+[e_2,e_1,e_2]=e_1+0-e_1=0, $\\
	any other possibility in the cyclic sum will have each of the component 0.

	\item 
\begin{scriptsize}
	\textbf{To Show:} $$[\alpha(a), \alpha(b),[c,d,e]]=[[a,b,c],\alpha(d),\alpha(e)]+[\alpha(c),[a,b,d],\alpha(e)]+[\alpha(c),\alpha(d),[a,b,e]],$$ $\forall a,b,c,d,e \in T$,\\
\end{scriptsize}

	\textbf{LHS:}\\ The only non 0 possibility for $[c,d,e]$ are $[e_1,e_2, e_2]=e_1$ and $[e_2,e_1, e_2]=-e_1$.\\
	In either of these possibilities LHS=0.\\
	
	\textbf{RHS:}\\
	\underline{Case i and ii}: When $(a,b)=(e_1,e_1)$ or $(e_2,e_2)$,\\ $[a,b,c]=[a,b,d]=[a,b,e]=0$, thus $RHS=0$.\\
	\underline {Case iii}: $(a,b)=(e_1,e_2)$\\
	Subcase i: $c=e_1$,\\
\begin{scriptsize}
 \begin{align*}
RHS &=[[e_1,e_2,e_1],\alpha(d),\alpha(e)]+[e_1,[e_1,e_2,d],\alpha(e)]+[\alpha(e_1),\alpha(d),[e_1,e_2,e]] \\
	 &=[0,\alpha(d),\alpha(e)]+[e_1,[e_1,e_2,d],\alpha(e)]+[\alpha(e_1),\alpha(d),[e_1,e_2,e]]\\
	 &=[e_1,[e_1,e_2,d],\alpha(e)]+[\alpha(e_1),\alpha(d),[e_1,e_2,e]]\\
       &=\begin{cases}
	 [e_1,0,e_1]+[e_1,e_1,0]=0+0=0 & \text{if}~ (d,e)=(e_1,e_1)\\
	 [e_1,e_1,-e_2]+[e_1,-e_2,e_1]=0+0=0& \text{if}~ (d,e)=(e_2,e_2)\\
	 [e_1,0,-e_2]+[e_1,e_1,e_1]=0+0=0 & \text{if}~ (d,e)=(e_1,e_2)\\
	 [e_1,e_1,e_1]+[e_1,-e_2,0]=0+0=0 & \text{if}~(d,e)=(e_2,e_1) \\ 
	 	\end{cases}
\end{align*}
\end{scriptsize}

%
	 
	Subcase ii: $c=e_2$\\
	   $RHS = 
	\begin{cases}
	0+0+0=0 & \text{if}~ (d,e)=(e_1,e_1)\\
	e_1-e_1+0=0& \text{if}~ (d,e)=(e_2,e_2)\\
	0+0+0=0 & \text{if}~ (d,e)=(e_1,e_2)\\
	0+0+0=0 & \text{if} ~(d,e)=(e_2,e_1) \\ 
	\end{cases}
	$
	
\underline {Case iv}: $(a,b)=(e_2,e_1)$\\
Subcase i: $c=e_1$,\\ 
$ 
RHS = 
\begin{cases}
0+0+0=0 & \text{if}~ (d,e)=(e_1,e_1)\\
0+0+0=0& \text{if}~ (d,e)=(e_2,e_2)\\
0+0+0=0 & \text{if}~ (d,e)=(e_1,e_2)\\
0+0+0=0 & \text{if}~ (d,e)=(e_2,e_1) \\ 
\end{cases}$\\
Subcase ii: $c=e_2$\\
$
RHS = 
\begin{cases}
0+0+0=0 & \text{if}~ (d,e)=(e_1,e_1)\\
-e_1+e_1+0=0& \text{if}~ (d,e)=(e_2,e_2)\\
0+0+0=0 & \text{if} ~(d,e)=(e_1,e_2)\\
0+0+0=0 & \text{if}~ (d,e)=(e_2,e_1) \\ 
\end{cases}
$	
\end{itemize}

Further note that 
$$\alpha[e_1,e_2,e_2]=\alpha(e_1)=e_1 ~\text{and}~ [\alpha(e_1),\alpha(e_2),\alpha(e_2)]=[e_1,-e_2,-e_2]=e_1$$
Similarly $\alpha[e_2,e_1,e_2]=[\alpha(e_2),\alpha(e_1),\alpha(e_2)]$.\\
Thus we have 
$$\alpha [abc]=[\alpha(a)\alpha(b)\alpha(c)], \forall a,b,c \in T$$

Hence $(T,[~],\alpha)$ is a multiplicative $Hom$-$Lts$.

\vspace{.5cm}

We now define group action $\mathbb Z_2 \times T\to T$ as 
$$0a=a ~\text{and} ~1 a=-a, ~\forall a\in T$$
For all $a,b,c\in T$, we have the following:
\begin{enumerate}
	\item $0a=a$
	\item 
	\begin{itemize}
		\item  $0(0a)=0(a)=a=(0+0)a$
		\item $1(1a)=1(-a)=a=0a=(1+1)a$
		\item $0(1a)=0(-a)=-a=1a=(0+1)a$
		\item $1(0a)=1(a)=-a=(1+0)a$
	\end{itemize}
	\item
	 \begin{itemize}
		\item $0(a+b)=a+b=0a+0b$
		\item $1(a+b)=-(a+b)=-a-b=1a+1b$
		\item $0(\lambda a)=\lambda a=\lambda(0a)$ for all scalars $\lambda$
		\item $1(\lambda a)=-\lambda a =\lambda (-a)=\lambda (1a)$ for all scalars $\lambda$
		\end{itemize}
		\item 
		\begin{itemize}
			\item $[0a,0b,0c]=[a,b,c]=0[a,b,c]$
			\item $[1a,1b,1c]=[-a,-b,-c]=-[a,b,c]=1[a,b,c]$
		\end{itemize}
		\item 
		\begin{itemize}
			\item $\alpha(0a)=\alpha(a)=0\alpha(a), \forall a\in T$
			\item $\alpha(1e_1)=\alpha(-e_1)=-\alpha(e_1)=-e_1$ and $1\alpha(e_1)=-\alpha(e_1)=-e_1$
			\item $\alpha(1e_2)=\alpha(-e_2)=-\alpha(e_2)=e_2$ and $1\alpha(e_2)=-\alpha(e_2)=e_2$
	
		\end{itemize}
		Thus we have $\alpha(ga)=g\alpha(a), \forall g\in G ~\text{and}~a\in T$.
\end{enumerate}

We have verified that the above is a well defined group action on $T$.

\vspace{.5cm}

On the above $Hom$-$Lts$ with the action of $\mathbb Z_2$, we now prove that the following is an equivariant deformation of $T$ of order 1:\\

Consider $\mu_t: T\otimes T\otimes T\to T$ such that $\mu_t=\mu_0+\mu_1 t$ where $\mu_0$ is the original ternary bracket $[~]$ and $\mu_1$ is a trilinear map defined as
$$\mu_1(e_2,e_1,e_1)=e_2 ~\text{and}~\mu_1(e_1,e_2,e_1)=-e_2,$$ and the remaining $\mu_1$ values are $0$. \\
Then $\forall a,b,c, d,e\in T$, we have
\begin{itemize}
	\item 
		$\mu_1(0a,0b,0c)=\mu_1(a,b,c)=0\mu_1(a,b,c)$ and\\ $\mu_1(1a,1b,1c)=\mu_1(-a,-b,-c)=-\mu_1(a,b,c)=1\mu_1(a,b,c)$\\
	i.e we have $\mu_1(ga,gb,gc)=g\mu_1(a,b,c)=0 ~\forall~ g\in G$

\item 
 $\mu_1(\alpha(e_2), \alpha(e_1),\alpha(e_1))=\mu_1(-e_2,e_1,e_1)=-e_2=\alpha(e_2)=\alpha \mu_1(e_2,e_1,e_1)$.\\
 We can similarly check for other combinations to get, $\mu_1\alpha=\alpha \mu_1$

\item By definition $\mu_1(a,a,b)=0$
	\item We note that\\
	$\mu_1(e_2,e_1,e_1)+\mu_1(e_1,e_1,e_2)+\mu_1(e_1,e_2,e_1)=e_2+0-e_2=0$\\
	any other combination of $e_1, e_2$ will lead to each component being 0.
	 Thus we have 
$$\mu_1(a,b,c)+\mu_1(b,c,a)+\mu_1(c,a,b)=0$$

\item Finally we need to prove 
 \begin{eqnarray}
\mu_t(\alpha(a),\alpha(b),\mu_t(c,d,e))=\mu_t(\mu_t(a,b,c),\alpha(d),\alpha(e))+\notag\\
\mu_t(\alpha(c),\mu_t(a,b,d),\alpha(e))+\mu_t(\alpha(c),\alpha(d),\mu_t(a,b,e)),\notag
\end{eqnarray}
for all  $a,b,c,d,e\in T$\\

Proving the above is equivalent to proving following three statements:

\textbf{Claim I:}
\begin{eqnarray}
&& [\alpha(a), \alpha(b),[c,d,e]]=[[a,b,c],\alpha(d),\alpha(e)]+[\alpha(c),[a,b,d],\alpha(e)]+\notag\\
&&[\alpha(c),\alpha(d),[a,b,e]], \forall a,b,c,d,e \in T \notag
\end{eqnarray} 

This has been verified earlier.\\

\textbf{Claim II:}
\begin{scriptsize}
\begin{eqnarray}
 &&[\alpha(a), \alpha(b),\mu_1(c,d,e)]+\mu_1(\alpha(a),\alpha(b),[c,d,e])\notag\\
&=&[\mu_1(a,b,c),\alpha(d),\alpha(e)]+\mu_1([a,b,c],\alpha(d),\alpha(e))+[\alpha(c), \mu_1(a,b,d), \alpha(e)]\notag\\
&&+\mu_1(\alpha(c),[a,b,d],\alpha(e))+[\alpha(c),\alpha(d),\mu_1(a,b,e)]+ \mu_1(\alpha(c),\alpha(d),[a,b,e]), \notag
\end{eqnarray} 
 for all  $a,b,c,d,e\in T$\\
\end{scriptsize}
\underline{Case i} $(a,b)=(e_1,e_1)$
\begin{eqnarray}
LHS&=&0+0=0\notag
\end{eqnarray} 
\begin{scriptsize}
\begin{eqnarray}
RHS&=&[\mu_1(e_1,e_1,c),\alpha(d),\alpha(e)]+\mu_1([e_1,e_1,c],\alpha(d),\alpha(e))+[\alpha(c), \mu_1(e_1,e_1,d), \alpha(e)]\notag\\
&&+\mu_1(\alpha(c),[e_1,e_1,d],\alpha(e))+[\alpha(c),\alpha(d),\mu_1(e_1,e_1,e)]+ \mu_1(\alpha(c),\alpha(d),[e_1,e_1,e])\notag\\
&=&0+0+0+0+0+0\notag\\
&=&0\notag
\end{eqnarray} 
\end{scriptsize}

\underline{Case ii} $(a,b)=(e_2,e_2)$
\begin{eqnarray}
LHS&=&0+0=0\notag
\end{eqnarray} 
\begin{scriptsize}
\begin{eqnarray}
RHS&=&[\mu_1(e_2,e_2,c),\alpha(d),\alpha(e)]+\mu_1([e_2,e_2,c],\alpha(d),\alpha(e))+[\alpha(c), \mu_1(e_2,e_2,d), \alpha(e)]\notag\\
&&+\mu_1(\alpha(c),[e_2,e_2,d],\alpha(e))+[\alpha(c),\alpha(d),\mu_1(e_2,e_2,e)]+ \mu_1(\alpha(c),\alpha(d),[e_2,e_2,e])\notag\\
&=&0+0+0+0+0+0\notag\\
&=&0\notag
\end{eqnarray} 
\end{scriptsize}
\underline{Case iii} $(a,b)=(e_1,e_2) $
\begin{eqnarray}
LHS&=&[e_1, -e_2,\mu_1(c,d,e)]+\mu_1(e_1,-e_2,[c,d,e])\notag
\end{eqnarray} 
\begin{scriptsize}
\begin{eqnarray}
RHS&=&[\mu_1(e_1,e_2,c),\alpha(d),\alpha(e)]+\mu_1([e_1,e_2,c],\alpha(d),\alpha(e))+[\alpha(c), \mu_1(e_1,e_2,d), \alpha(e)]\notag\\
&&+\mu_1(\alpha(c),[e_1,e_2,d],\alpha(e))+[\alpha(c),\alpha(d),\mu_1(e_1,e_2,e)]+ \mu_1(\alpha(c),\alpha(d),[e_1,e_2,e])\notag
\end{eqnarray} 
\end{scriptsize}
Subcase i:  $(c,d)=(e_1,e_1)$,
\begin{eqnarray}
LHS&=&[e_1, -e_2,\mu_1(e_1,e_1,e)]+\mu_1(e_1,-e_2,[e_1,e_1,e])=0+0=0\notag
\end{eqnarray} 
\begin{scriptsize}
\begin{eqnarray}
RHS&=&[\mu_1(e_1,e_2,e_1),e_1,\alpha(e)]+\mu_1([e_1,e_2,e_1],e_1,\alpha(e))+[e_1, \mu_1(e_1,e_2,e_1), \alpha(e)]\notag\\
&&+\mu_1(e_1,[e_1,e_2,e_1],\alpha(e))+[e_1,e_1,\mu_1(e_1,e_2,e)]+ \mu_1(e_1,e_1,[e_1,e_2,e])\notag\\
&=&[-e_2,e_1,\alpha(e)]+0+[e_1,-e_2,\alpha(e)]+0+0+0\notag\\
&=&
\begin{cases}
0+0=0 & \text{if}~e=e_1~~~~~~~~~~~~~~~~~~~~~~~~~~~~~~~~~~~~~~~~~~~~~~~~~~\\
-e_1+e_1=0 & \text{if}~e=e_2~~~~~~~~~~~~~~~~~~~~~~~~~~~~~~~~~~~~~~~~~~~~~~~~~~~~~~\\
\end{cases}\notag
\end{eqnarray}
\end{scriptsize} 

Subcase ii:  $(c,d)=(e_2,e_2)$,
\begin{eqnarray}
LHS&=&[e_1, -e_2,\mu_1(e_2,e_2,e)]+\mu_1(e_1,-e_2,[e_2,e_2,e])=0+0=0\notag
\end{eqnarray} 
\begin{scriptsize}
\begin{eqnarray}
RHS&=&[\mu_1(e_1,e_2,e_2),-e_2,\alpha(e)]+\mu_1([e_1,e_2,e_2],-e_2,\alpha(e))+[-e_2, \mu_1(e_1,e_2,e_2), \alpha(e)]\notag\\
&&+\mu_1(-e_2,[e_1,e_2,e_2],\alpha(e))+[-e_2,-e_2,\mu_1(e_1,e_2,e)]+ \mu_1(-e_2,-e_2,[e_1,e_2,e])\notag\\
&=&0+\mu_1(e_1,-e_2,\alpha(e))+0+\mu_1(-e_2,e_1,\alpha(e))+0+0\notag\\
&=&
\begin{cases}
e_2-e_2=0 & \text{if}~e=e_1~~~~~~~~~~~~~~~~~~~~~~~~~~~~~~~~~~~~~~~~~~~~~~~~~~~~~~\\
0+0=0 & \text{if}~e=e_2~~~~~~~~~~~~~~~~~~~~~~~~~~~~~~~~~~~~~~~~~~~~~~~~~~~~~~~~~~~~~~\\
\end{cases}\notag
\end{eqnarray} 
\end{scriptsize}

Subcase iii:  $(c,d)=(e_1,e_2)$,
\begin{eqnarray}
LHS&=&[e_1, -e_2,\mu_1(e_1,e_2,e)]+\mu_1(e_1,-e_2,[e_1,e_2,e])\notag\\
&=&
\begin{cases}
[e_1, -e_2,\mu_1(e_1,e_2,e_1)]+\mu_1(e_1,-e_2,[e_1,e_2,e_1]) &\text{if}~e=e_1\\
[e_1, -e_2,\mu_1(e_1,e_2,e_2)]+\mu_1(e_1,-e_2,[e_1,e_2,e_2]) &\text{if}~e=e_2\\
\end{cases}\notag\\
&=&
\begin{cases}
[e_1,-e_2,-e_2]+0=e_1 &\text{if}~e=e_1\\
0+\mu_1(e_1,-e_2,e_1)=e_2 &\text{if}~e=e_2\\
\end{cases}\notag
\end{eqnarray}
\begin{scriptsize}
\begin{eqnarray}
RHS&=&[\mu_1(e_1,e_2,e_1),-e_2,\alpha(e)]+\mu_1([e_1,e_2,e_1],-e_2,\alpha(e))+[e_1, \mu_1(e_1,e_2,e_2), \alpha(e)]\notag\\
&&+\mu_1(e_1,[e_1,e_2,e_2],\alpha(e))+[e_1,-e_2,\mu_1(e_1,e_2,e)]+ \mu_1(e_1,-e_2,[e_1,e_2,e])\notag \\
&=&[-e_2,-e_2,\alpha(e)]+\mu_1(0,-e_2,\alpha(e))+[e_1,0, \alpha(e)]+\notag\\
&&\mu_1(e_1,e_1,\alpha(e))+[e_1,-e_2,\mu_1(e_1,e_2,e)]+ \mu_1(e_1,-e_2,[e_1,e_2,e]\notag\\
&=&0+0+0+0+[e_1,-e_2,\mu_1(e_1,e_2,e)]+ \mu_1(e_1,-e_2,[e_1,e_2,e])\notag\\
&=&
\begin{cases}
[e_1,-e_2,\mu_1(e_1,e_2,e_1)]+ \mu_1(e_1,-e_2,0) &\text{if}~e=e_1\\
[e_1,-e_2,0]+ \mu_1(e_1,-e_2,e_1) &\text{if}~e=e_2\\
\end{cases} \notag\\
&=&
\begin{cases}
[e_1,-e_2,-e_2]+0=e_1 &\text{if}~e=e_1\\
0+e_2=e_2 &\text{if}~e=e_2\\
\end{cases} \notag
\end{eqnarray} 
\end{scriptsize}

Subcase iv:  $(c,d)=(e_2,e_1)$,
\begin{eqnarray}
LHS&=&[e_1,-e_2,\mu_1(e_2,e_1,e)]+\mu_1(e_1,-e_2,[e_2,e_1,e])\notag\\
&=&
\begin{cases}
[e_1, -e_2,\mu_1(e_2,e_1,e_1)]+\mu_1(e_1,-e_2,[e_2,e_1,e_1]) &\text{if}~e=e_1\\
[e_1, -e_2,\mu_1(e_2,e_1,e_2)]+\mu_1(e_1,-e_2,[e_2,e_1,e_2]) &\text{if}~e=e_2
\end{cases}\notag\\
&=&
\begin{cases}
[e_1,-e_2,e_2]+0=-e_1 &\text{if}~e=e_1\\
0+\mu_1(e_1,-e_2,-e_1)=-e_2 &\text{if}~e=e_2
\end{cases}\notag
\end{eqnarray}
\begin{scriptsize}
\begin{eqnarray}
RHS&=&[\mu_1(e_1,e_2,e_2),e_1,\alpha(e)]+\mu_1([e_1,e_2,e_2],e_1,\alpha(e))+[-e_2, \mu_1(e_1,e_2,e_1), \alpha(e)]\notag\\
&&+\mu_1(-e_2,[e_1,e_2,e_1],\alpha(e))+[-e_2,e_1,\mu_1(e_1,e_2,e)]+ \mu_1(-e_2,e_1,[e_1,e_2,e])\notag\\
&=&[0,e_1,\alpha(e)]+\mu_1(e_1,-e_1,\alpha(e))+[-e_2, -e_2, \alpha(e)]+\notag\\
&&\mu_1(-e_2,0,\alpha(e))+[-e_2,e_1,\mu_1(e_1,e_2,e)]+ \mu_1(-e_2,e_1,[e_1,e_2,e])\notag\\
&=&0+0+0+0+[-e_2,e_1,\mu_1(e_1,e_2,e)]+ \mu_1(-e_2,e_1,[e_1,e_2,e])\notag\\
&=&
\begin{cases}
[-e_2, e_1,-e_2]+0=-e_1 &\text{if}~e=e_1~~~~~~~~~~~~~~~~~~~~~~~~~~~~~~~~~~\\
0+\mu_1(-e_2,e_1,e_1)=-e_2 &\text{if}~e=e_2~~~~~~~~~~~~~~~~~~~~~~~~~~~~~~~~~~~~~
\end{cases}\notag
\end{eqnarray}
\end{scriptsize}
\underline{Case iv}  $(a,b)=(e_2,e_1) $\\
$LHS=[-e_2, e_1,\mu_1(c,d,e)]+\mu_1(-e_2,e_1,[c,d,e])$
\begin{scriptsize}
\begin{eqnarray}
RHS&=&[\mu_1(e_2,e_1,c),\alpha(d),\alpha(e)]+\mu_1([e_2,e_1,c],\alpha(d),\alpha(e))+[\alpha(c), \mu_1(e_1,e_1,d), \alpha(e)]\notag\\
&&+\mu_1(\alpha(c),[e_2,e_1,d],\alpha(e))+[\alpha(c),\alpha(d),\mu_1(e_2,e_1,e)]+ \mu_1(\alpha(c),\alpha(d),[e_2,e_1,e])\notag
\end{eqnarray}
\end{scriptsize}
Subcase i:  $(c,d)=(e_1,e_1)$,\\
$LHS=[e_2, -e_1,\mu_1(e_1,e_1,e)]+\mu_1(e_2,-e_1,[e_1,e_1,e])=0+0=0$
\begin{scriptsize}
\begin{eqnarray}
RHS&=&[\mu_1(e_2,e_1,e_1),e_1,\alpha(e)]+\mu_1([e_2,e_1,e_1],e_1,\alpha(e))+[e_1, \mu_1(e_1,e_1,e_1), \alpha(e)]\notag\\
&&+\mu_1(e_1,[e_2,e_1,e_1],\alpha(e))+[e_1,e_1,\mu_1(e_2,e_1,e)]+ \mu_1(e_1,e_1,[e_2,e_1,e])\notag\\
&=&[e_2,e_1,\alpha(e)]+\mu_1(0,e_1,\alpha(e))+[e_1,e_2, \alpha(e)]+\notag\\
&&\mu_1(e_1,0,\alpha(e))+[e_1,e_1,\mu_1(e_2,e_1,e)]+ \mu_1(e_1,e_1,[e_2,e_1,e])\notag\\
&=&[e_2,e_1,\alpha(e)]+0+[e_1,e_2, \alpha(e)]+0+0+0\notag\\
&=&[e_2,e_1,\alpha(e)]-[e_2,e_1, \alpha(e)]=0\notag
\end{eqnarray}
\end{scriptsize}
Subcase ii:  $(c,d)=(e_2,e_2)$,\\
$LHS=[e_2, -e_1,\mu_1(e_2,e_2,e)]+\mu_1(e_2,-e_1,[e_2,e_2,e])=0+0=0$
\begin{scriptsize}
\begin{eqnarray}
RHS&=&[\mu_1(e_2,e_1,e_2),-e_2,\alpha(e)]+\mu_1([e_2,e_1,e_2],-e_2,\alpha(e))+[-e_2, \mu_1(e_1,e_1,e_2), \alpha(e)]\notag\\
&&+\mu_1(-e_2,[e_2,e_1,e_2],\alpha(e))+[-e_2,-e_2,\mu_1(e_2,e_1,e)]+ \mu_1(-e_2,-e_2,[e_2,e_1,e])\notag\\
&=&[0,-e_2,\alpha(e)]+\mu_1(-e_1,-e_2,\alpha(e))+[-e_2,0, \alpha(e)]+\notag\\
&&\mu_1(-e_2,-e_1,\alpha(e))+0+ 0\notag\\
&=&0+\mu_1(-e_1,-e_2,\alpha(e))+0+\mu_1(-e_2,-e_1,\alpha(e))+0+ 0\notag\\
&=&\mu_1(-e_1,-e_2,\alpha(e))-\mu_1(-e_1,-e_2,\alpha(e))\notag\\
&=&0\notag
\end{eqnarray}
\end{scriptsize}
Subcase iii:  $(c,d)=(e_1,e_2)$,
\begin{eqnarray}
LHS&=&[-e_2, e_1,\mu_1(e_1,e_2,e)]+\mu_1(-e_2,e_1,[e_1,e_2,e])\notag\\
&=&
\begin{cases}
[-e_2, e_1,\mu_1(e_1,e_2,e_1)]+\mu_1(-e_2,e_1,[e_1,e_2,e_1])&\text{if}~e=e_1\\
[-e_2, e_1,\mu_1(e_1,e_2,e_2)]+\mu_1(-e_2,e_1,[e_1,e_2,e_2])&\text{if}~e=e_2
\end{cases}\notag\\
&=&
\begin{cases}
[-e_2,e_1,-e_2]+0=-e_1 &\text{if}~e=e_1\\
0+\mu_1(-e_2,e_1,e_1)=-e_2 &\text{if}~e=e_2
\end{cases}\notag
\end{eqnarray}
\begin{scriptsize}
\begin{eqnarray}
RHS&=&[\mu_1(e_2,e_1,e_1),-e_2,\alpha(e)]+\mu_1([e_2,e_1,e_1],-e_2,\alpha(e))+[e_1, \mu_1(e_1,e_1,e_2), \alpha(e)]\notag\\
&&+\mu_1(e_1,[e_2,e_1,e_2],\alpha(e))+[e_1,-e_2,\mu_1(e_2,e_1,e)]+ \mu_1(e_1,-e_2,[e_2,e_1,e])\notag\\
&=&[-e_2,-e_2,\alpha(e)]+\mu_1(0,-e_2,\alpha(e))+[e_1,0, \alpha(e)]+\notag\\
&&\mu_1(e_1,-e_1,\alpha(e))+[e_1,-e_2,\mu_1(e_2,e_1,e)]+ \mu_1(e_1,-e_2,[e_2,e_1,e])\notag\\
&=&0+0+0+0+[e_1,-e_2,\mu_1(e_2,e_1,e)]+ \mu_1(e_1,-e_2,[e_2,e_1,e])\notag\\
&=&
\begin{cases}
[e_1,-e_2,e_2]+\mu_1(e_1,-e_2,0) &\text{if}~e=e_1~~\\
[e_1,-e_2,0]+ \mu_1(e_1,-e_2,-e_1) &\text{if}~e=e_2~~~~\\
\end{cases}\notag\\
&=&
\begin{cases}
-e_1+0=-e_1 &\text{if}~e=e_1~~\\
0+-e_2=-e_2 &\text{if}~e=e_2~~~~\\
\end{cases}\notag
\end{eqnarray}
\end{scriptsize}

Subcase iv:  $(c,d)=(e_2,e_1)$,
\begin{eqnarray}
LHS&=&[-e_2,e_1,\mu_1(e_2,e_1,e)]+\mu_1(-e_2,e_1,[e_2,e_1,e])\notag\\
&=&
\begin{cases}
[-e_2, e_1,\mu_1(e_2,e_1,e_1)]+\mu_1(-e_2,e_1,[e_2,e_1,e_1]) &\text{if}~e=e_1\\
[-e_2, e_1,\mu_1(e_2,e_1,e_2)]+\mu_1(-e_2,e_1,[e_2,e_1,e_2]) &\text{if}~e=e_2\\
\end{cases}\notag\\
&=&
\begin{cases}
[-e_2,e_1,e_2]+0=e_1 &\text{if}~e=e_1\\
0+\mu_1(-e_2,e_1,-e_1)=e_2 &\text{if}~e=e_2\\
\end{cases}\notag
\end{eqnarray}
\begin{scriptsize}
\begin{eqnarray}
RHS&=&[\mu_1(e_2,e_1,e_2),e_1,\alpha(e)]+\mu_1([e_2,e_1,e_2],e_1,\alpha(e))+[-e_2, \mu_1(e_1,e_1,e_1), \alpha(e)]\notag\\
&&+\mu_1(-e_2,[e_2,e_1,e_1],\alpha(e))+[-e_2,e_1,\mu_1(e_2,e_1,e)]+ \mu_1(-e_2,e_1,[e_2,e_1,e])\notag\\
&=&[0,e_1,\alpha(e)]+\mu_1(-e_1,e_1,\alpha(e))+[-e_2, e_2, \alpha(e)]+\notag\\
&&\mu_1(-e_2,0,\alpha(e))+[-e_2,e_1,\mu_1(e_2,e_1,e)]+ \mu_1(-e_2,e_1,[e_2,e_1,e])\notag\\
&=&0+0+0+0+[-e_2,e_1,\mu_1(e_2,e_1,e)]+ \mu_1(-e_2,e_1,[e_2,e_1,e])\notag\\
&=&
\begin{cases}
[-e_2,e_1,e_2]+ \mu_1(-e_2,e_1,0) &\text{if}~e=e_1~~~~~~~~~~~~~\\
[-e_2,e_1,0]+ \mu_1(-e_2,e_1,-e_1) &\text{if}~e=e_2~~~~~~~~~~~~ \\
\end{cases}\notag\\
&=&
\begin{cases}
e_1+0=e_1&\text{if}~e=e_1~~~~~~~~~~~~~\\
0+e_2=e_2&\text{if}~e=e_2~~~~~~~~~~~~ \\
\end{cases}\notag
\end{eqnarray}
\end{scriptsize}


 \textbf{Claim III:} 
\begin{scriptsize}
\begin{eqnarray}
\mu_1(\alpha(a), \alpha(b), \mu_1(c,d,e))&=&\mu_1(\mu_1(a,b,c),\alpha(d), \alpha(e))+\mu_1(\alpha(c),\mu_1(a,b,d),\alpha(e)]+\notag\\
&&\mu_1(\alpha(c),\alpha(d),\mu_1(a,b,e))~ \text{for all}~  a,b,c,d,e \in T\notag 
 \end{eqnarray}
\end{scriptsize}
\underline{Case i and ii} When $(a,b)=(e_1,e_1) $ or $(e_2,e_2)$, it is easy to see that \\
$RHS=LHS=0$\\
\underline{Case iii} $(a,b)=(e_1,e_2) $ 
\begin{eqnarray}
LHS &=& \mu(\alpha(e_1), \alpha(e_2), \mu_1(c,d,e))=\mu(e_1, -e_2, \mu_1(c,d,e))\notag\\
RHS&=&\mu_1(\mu_1(e_1,e_2,c),\alpha(d), \alpha(e))+\mu_1(\alpha(c),\mu_1(e_1,e_2,d),\alpha(e)]+\notag\\
&&\mu_1(\alpha(c),\alpha(d),\mu_1(e_1,e_2,e))\notag
\end{eqnarray}

Subcase i:  $(c,d)=(e_1,e_1)$,
\begin{eqnarray}
LHS&=&0\notag\\
RHS&=&\mu_1(\mu_1(e_1,e_2,e_1),e_1, \alpha(e))+\mu_1(e_1,\mu_1(e_1,e_2,e_1),\alpha(e)]+\notag\\
&&\mu_1(e_1,e_1,\mu_1(e_1,e_2,e))\notag\\
&=&
\begin{cases}
\mu_1(-e_2,e_1, e_1)+\mu_1(e_1,-e_2,e_1)+\mu_1(e_1,e_1,-e_2,) &\text{if}~ e=e_1\\
\mu_1(-e_2,e_1,-e_2)+\mu_1(e_1,-e_2,-e_2)+\mu_1(e_1,e_1,0) &\text{if}~ e=e_2\\
\end{cases}\notag\\
&=&
\begin{cases}
-e_2+e_2+0=0&\text{if}~ e=e_1\\
0+0+0=0 &\text{if}~ e=e_2\\
\end{cases}\notag
\end{eqnarray}

Subcase ii: $(c,d)=(e_2,e_2)$,
\begin{eqnarray}
 LHS&=&0\notag\\
RHS&=&0+0+0\notag
\end{eqnarray}

Subcase iii: $(c,d)=(e_1,e_2)$
\begin{eqnarray}
LHS&=&\mu_1(e_1, -e_2, \mu_(e_1, e_2,e))\notag\\
&=&
\begin{cases}
\mu_1(e_1,-e_2, -e_2)=0&\text{if}~ e=e_1~~~~~~~~~~~~~~~~~~~~~~~~~~~~~~~~~~~~~~\\
 \mu_1(e_1,-e_2, 0)=0&\text{if}~ e=e_2\\
\end{cases}\notag\\
RHS&=&0+0+\mu_1(e_1,-e_2,\mu_1(e_1,e_2,e))\notag\\
&=&\mu_1(e_1,-e_2,\mu_1(e_1,e_2,e))\notag\\
&=&
\begin{cases}
\mu_1(e_1,-e_2,\mu_1(e_1,e_2,e_1))=\mu_1(e_1,-e_2,-e_2)=0&\text{if}~ e=e_1\\
\mu_1(e_1,-e_2, \mu_1(e_1,e_2,e_2))=\mu_1(e_1,-e_2,0)=0&\text{if}~ e=e_2\\
\end{cases}\notag
\end{eqnarray}

 Subcase iv:  $(c,d)=(e_2,e_1)$
\begin{eqnarray}
 LHS&=&\mu_1(e_1, -e_2, \mu_1(e_2, e_1,e))\notag\\
&=&
\begin{cases}
\mu_1(e_1,-e_2, e_2)=0&\text{if}~ e=e_1~~~~~~~~~~~~~~~~~~~~~~~~~~~~~~~~~~~~~~~~~~~~~\\
 \mu_1(e_1,-e_2, 0)=0&\text{if}~ e=e_2\notag\\
\end{cases}\notag\\
RHS&=&
\begin{cases}
0+0+\mu_1(-e_2,e_1, \mu_1(e_1,e_2,e_1))&\text{if}~e=e_1\\
0+0+\mu_1(-e_2,e_1, \mu_1(e_1,e_2,e_2)) &\text{if} ~e=e_2\\
\end{cases}\notag\\
&=&
\begin{cases}
\mu_1(-e_2,e_1,-e_2)=0&\text{if}~e=e_1\\
 \mu_1(-e_2,e_1,0)=0=0&\text{if} ~e=e_2\\
\end{cases}\notag
\end{eqnarray}

\underline{Case iv} $(a,b)=(e_2,e_1) $ 
\begin{eqnarray}
LHS &=& \mu_1(\alpha(e_2), \alpha(e_1), \mu_1(c,d,e))\notag\\
&=&\mu_1(-e_2, e_1, \mu_1(c,d,e))\notag\\
RHS&=&\mu_1(\mu_1(e_2,e_1,c),\alpha(d), \alpha(e))+\mu_1(\alpha(c),\mu_1(e_2,e_1,d),\alpha(e)]+\notag\\
&&\mu_1(\alpha(c),\alpha(d),\mu_1(e_2,e_1,e))\notag
\end{eqnarray}

Subcase i:  $(c,d)=(e_1,e_1)$,
\begin{eqnarray}
LHS&=&0\notag\\
RHS&=&\mu_1(\mu_1(e_2,e_1,e_1),e_1, \alpha(e))+\mu_1(e_1,\mu_1(e_2,e_1,e_1),\alpha(e)]+\notag\\
&&\mu_1(\alpha(e_1),\alpha(e_1),\mu_1(e_2,e_1,e))\notag\\
&=&
\begin{cases}
\mu_1(e_2,e_1,e_1)+\mu_1(e_1,e_2,e_1)+\mu_1(e_1,e_1,e_2)&\text{if}~ e=e_1\\
\mu_1(e_2,e_1, -e_2)+\mu_1(e_1,e_2,-e_2)+\mu_1(e_1,e_1,0) &\text{if}~ e=e_2\\
\end{cases}\notag\\
&=&
\begin{cases}
e_2-e_2+0=0&\text{if}~ e=e_1\\
0+0+0=0 &\text{if}~ e=e_2\\
\end{cases}\notag
\end{eqnarray}

Subcase ii: $(c,d)=(e_2,e_2)$,
\begin{eqnarray}
LHS&=&0\notag\\
RHS&=&\mu_1(\mu_1(e_2,e_1,e_2),-e_2,\alpha(e))+\mu_1(-e_2,\mu_1(e_2,e_1,e_2),\alpha(e))+\notag\\
&&\mu_1(-e_2,-e_2,\mu_1(e_2,e_1,e))\notag\\
&=&   \mu_1(0,-e_2,\alpha(e))+\mu_1(-e_2,0,\alpha(e))+0 \notag\\
&=&0\notag
\end{eqnarray}

Subcase iii: $(c,d)=(e_1,e_2)$
\begin{eqnarray}
LHS&=&\mu_1(-e_2, e_1, \mu_(e_1, e_2,e)\notag\\
&=&
\begin{cases}
\mu_1(-e_2,e_1,-e_2)=0&\text{if}~e=e_1~~~~~~~~~~~~~~~~~~~~~~~~~~~~~~~~~~~~~~\\
\mu_1(-e_2,e_1, 0)=0&\text{if}~e=e_2\\
\end{cases}\notag\\	
RHS&=&\mu_1(\mu_1(e_2,e_1,e_1),-e_2, \alpha(e))+\mu_1(e_1,\mu_1(e_2,e_1,e_2),\alpha(e))+\notag\\
&&\mu_1(e_1,-e_2,\mu_1(e_2,e_1,e))\notag\\
&=&\mu_1(e_2,-e_2,\alpha(e))+0+\mu_1(e_1,-e_2,\mu_1(e_2,e_1,e))\notag\\
&=&
\begin{cases}
\mu_1(e_1,-e_2,\mu_1(e_2,e_1,e_1))&\text{if}~ e=e_1\\
\mu_1(e_1,-e_2, \mu_1(e_2,e_1,e))&\text{if}~ e=e_2\\
\end{cases}\notag\\
&=&
\begin{cases}
\mu_1(e_1,-e_2,e_2)=0&\text{if}~ e=e_1\\
\mu_1(e_1,-e_2,0)=0&\text{if} ~e=e_2\\
\end{cases}\notag
\end{eqnarray}

Subcase iv:  $(c,d)=(e_2,e_1)$,
\begin{eqnarray}
LHS&=&\mu_1(-e_2, e_1, \mu_(e_2, e_1,e)\notag\\
&=&
\begin{cases}
\mu_1(-e_2,e_1, e_2)=0&\text{if}~ e=e_1~~~~~~~~~~~~~~~~~~~~~~~~~~~~~~~~~~~~~~~~~~~~\\
\mu_1(-e_2,e_1, 0)=0&\text{if}~ e=e_2\\
\end{cases}\notag\\
RHS&=&\mu_1(\mu_1(e_2,e_1,e_2),\alpha(e_1), \alpha(e))+\mu_1(\alpha(e_2),\mu_1(e_2,e_1,e_1),\alpha(e))+\notag\\
&&\mu_1(\alpha(e_2),\alpha(e_1),\mu_1(e_2,e_1,e))\notag\\
&=&\mu_1(0,e_1, \alpha(e))+\mu_1(-e_2,e_2,\alpha(e))+\mu_1(-e_2,e_1,\mu_1(e_2,e_1,e))\notag\\
&=&0+0+\mu_1(-e_2,e_1,\mu_1(e_2,e_1,e))\notag\\
&=&\mu_1(-e_2,e_1,\mu_1(e_2,e_1,e))\notag\\
&=&
\begin{cases}
\mu_1(-e_2,e_1, \mu_1(e_2,e_1,e_1))&\text{if}~e=e_1\\
\mu_1(-e_2,e_1, \mu_1(e_2,e_1,e_2))&\text{if} ~e=e_2\\
\end{cases}\notag\\
&=&
\begin{cases}
 \mu_1(-e_2,e_1,e_2)=0&\text{if}~e=e_1\\
 \mu_1(-e_2,e_1,0)=0&\text{if} ~e=e_2\\
\end{cases}\notag
\end{eqnarray}

\end{itemize}

\end{eg}

%
%

\end{document}